\documentclass[11pt,twoside]{article}
\usepackage[margin=1.3in]{geometry}

\usepackage{amsfonts,amsmath,amssymb,amsthm}
\usepackage{CJK,CJKnumb,CJKulem,cases,cite,color,curves}
\usepackage{bm,dsfont,enumitem,extarrows,fontenc,graphicx,ifthen,latexsym,listings,makeidx,mathrsfs,psfrag,times,xcolor}
\usepackage[colorlinks,citecolor=black,linkcolor=black]{hyperref}
\usepackage[title]{appendix}

\let\oldbibliography\thebibliography
\renewcommand{\thebibliography}[1]{\oldbibliography{#1}\setlength{\itemsep}{0pt}}

\makeindex
\newtheorem{theorem}{Theorem}[section]
\newtheorem{lemma}{Lemma}[section]

\newtheorem{remark}{Remark}[section]

\newcommand{\R}{\mathbb R}

\begin{document}
\title{\textbf{On isolated singularities of the conformal Gaussian curvature equation and $Q$-curvature equation}\bigskip}

\author{\medskip Hui Yang\footnote{H. Yang is partially supported by NSFC grant 12301140 and the Institute of Modern Analysis-A Frontier Research Center of Shanghai.} \quad and \quad Ronghao Yang}

\date{\today} 

\maketitle

\begin{center}
\begin{minipage}{130mm} 

\begin{center}{\bf Abstract}\end{center}
In this paper, we study the isolated singularities of the conformal Gaussian curvature equation
\[
 -\Delta u = K(x) e^{u} \quad \textmd{in } B_{1} \setminus \{ 0 \}, 
\] 
where $B_1 \setminus \{ 0 \} \subset \mathbb{R}^2$ is the punctured unit disc. Under the assumption that the Gaussian curvature $K \in L^\infty(B_1)$ is nonnegative, we establish the asymptotic behavior of solutions near the singularity. When $K \equiv 1$, a similar result has been obtained by Chou and Wan (Pacific J. Math. 1994) using the method of complex analysis. Our proof is entirely based on the PDE method and applies to the general Gaussian curvature $K(x)$. Furthermore, our approach is also available for characterizing isolated singularities of the conformal $Q$-curvature equation $(-\Delta)^{\frac{n}{2}} u = K(x) e^{u}$ in any dimension $n\geq 3$. This equation arises from the prescribing $Q$-curvature problem.   

\medskip 

\noindent{\it Keywords}: Isolated singularities, conformal Gaussian curvature equation, conformal $Q$-curvature equation 

\medskip

\noindent {\it MSC (2020)}: 35J91; 35B40; 35A21   

\end{minipage}
\end{center}

\section{Introduction} 

Let $(M, g)$ be a two-dimensional Riemannian manifold and $K$ be a given function on $M$. The Nirenberg problem (or, prescribing Gaussian curvature problem) is the following: does there exist a metric $g_1$ conformal to $g$ on $M$ such that $K/2$ is the Gaussian curvature of the new metric $g_1$? If we write $g_1 = e^{u} g$, then the problem is equivalent to solving the following elliptic equation  
\begin{equation}\label{mfdeq}
    -\Delta_{g}u + 2k = Ke^{u} \quad \textmd{in } M,  
\end{equation}
where $\Delta_{g}$ is the Laplace-Beltrami operator on $M$ and $k$ is the Guassian curvature of $g$. When $M$ is compact or noncompact, the existence and nonexistence of \eqref{mfdeq} have been extensively studied, see, for example, \cite{CY,CLi2,CLi3,CL,CN,H,K,M} and the references therein. When $(M,g)$ is complete and noncompact, a natural model is $M= \R^{2}$ and $g$ is the standard Euclidean metric. In this case, the above equation reduces to 
\begin{equation}\label{R2}
    -\Delta u = K(x) e^{u} \quad \textmd{in } \R^{2}. 
\end{equation} 

When $K \equiv 1$, under the assumption that   
\begin{equation}\label{la}
    \int_{\R^{2}}e^{u(x)}dx<\infty, 
\end{equation}
Chen and Li \cite{CLi} established the classification of solutions of \eqref{R2} using the moving planes method. This classification result is significant for understanding the structure of the solutions set of the Gaussian curvature equation. To study \eqref{R2}, a fundamental problem is to characterize the asymptotic behavior of solutions at infinity. Using the Kelvin transform, the asymptotic behavior at infinity can be reduced to the following isolated singularity problem  
\begin{equation}\label{m=1eq}
    -\Delta u = K(x) e^{u} \quad \textmd{in } B_{1} \setminus \left \{ 0 \right \} \subset \mathbb{R}^2,  
\end{equation}
where $B_1 \setminus \{ 0 \}$ is the punctured unit disc. On the other hand, equation \eqref{m=1eq} also appears in the study of prescribing Gaussian curvature problem on a Riemannian surface with conical singularities (see, e.g., Chen-Li \cite{CLi3}). Assume that $K \equiv 1$ and 
\begin{equation}\label{la=abc}  
    \int_{B_1\setminus \{ 0 \}} e^{u(x)}dx< +\infty, 
\end{equation}
Chou and Wan \cite{CW} applied the method of complex analysis to show that there exists a constant $\alpha > -2$ such that 
\begin{equation}\label{jdgj}
u(x) = \alpha \ln |x| + O(1) \quad \textmd{as }  x \to 0. 
\end{equation}
Using the asymptotic estimate \eqref{jdgj}, Guo, Wan and Yang \cite{GWY} established an asymptotic expansion for $u$ near the origin. 

A natural question is: can one give a proof for the asymptotic behavior \eqref{jdgj} using the PDE method? Furthermore, under what conditions about $K(x)$, does the asymptotic behavior \eqref{jdgj} still hold for the equation \eqref{m=1eq}? In this paper, we will develop a PDE approach to study the asymptotic behavior of solutions of \eqref{m=1eq} near the isolated singularity. Under the assumption that $K \in L^\infty(B_1)$ is nonnegative, we have the following result.  

\begin{theorem}\label{m=1th}
Suppose that $K \in L^\infty(B_{1})$ is nonnegative and $u \in C^{2}(B_1 \setminus \{0\})$ is a solution of \eqref{m=1eq}. If $u$ satisfies \eqref{la=abc}, 
then there exists a constant $\alpha > -2$ and a H\"{o}lder continuous function $\varphi \in C^\gamma_{\textmd{loc}}(B_1)$ ($0< \gamma<1$) such that 
\begin{equation}\label{asy=01}
u(x) = \alpha \ln |x| + \varphi(x)  \quad  \textmd{near the origin}.  
\end{equation}
\end{theorem}

\begin{remark}\label{rem25}
(i) By the Kelvin transform $u(x) = v\big( \frac{x}{|x|^2} \big) - 4 \ln |x|$, Theorem \ref{m=1th} can be used to characterize the asymptotic behavior of solutions of the Gaussian curvature equation in an exterior domain
\begin{equation}\label{equinty}
-\Delta v = K(x) e^{v} \quad \textmd{in } \mathbb{R}^2 \setminus B_{1},  ~~~ \int_{\mathbb{R}^2 \setminus B_{1}} e^{v(x)} dx < \infty. 
\end{equation}
Let $K \in L^\infty(\mathbb{R}^2 \setminus B_{1})$ be nonnegative and $v \in C^2(\mathbb{R}^2 \setminus B_{1})$ be a solution of \eqref{equinty}. Then we have the following result: there exists a constant $\beta < -2$ such that $v(x) = \beta \ln |x| +  O(1)$ near infinity.

(ii) Theorem \ref{m=1th} can also be applied to unbounded function $K(x)$ with some growth restrictions at the origin. For example, suppose that $u \in C^{2}(B_1 \setminus \{0\})$ is a solution of 
\begin{equation}\label{equinty=02}
-\Delta u = |x|^{-\gamma} K(x) e^{u} \quad \textmd{in } B_{1} \setminus \{0\},  ~~ \int_{B_{1} \setminus \{0\}} |x|^{-\gamma} e^{u(x)} dx < \infty, 
\end{equation}
where $\gamma >0$ and $K \in L^\infty(B_{1})$  is nonnegative. Let $w(x)= u(x) - \gamma \ln |x|$. By Theorem \ref{m=1th}, there exists $\alpha > -2$ such that $w(x) = \alpha \ln |x| + \varphi(x)$ for some $\varphi \in C^\gamma_{\textmd{loc}}(B_1)$. Hence, $u(x) = (\alpha+\gamma) \ln |x| + \varphi(x)$ near the origin.  
\end{remark}

It is interesting to contrast Theorem \ref{m=1th} and the corresponding higher-dimensional result. For $n\geq 3$, the higher-dimensional analogue of \eqref{m=1eq} is the following conformal scalar curvature equation  
\begin{equation}\label{scalar=01} 
    -\Delta u = K(x) u^{\frac{n+2}{n-2}}, \quad u > 0  \quad \textmd{in } B_{1} \setminus \left \{ 0 \right \} \subset \mathbb{R}^n. 
\end{equation} 
Every solution of \eqref{scalar=01} induces a metric conformal to the flat metric of $\mathbb{R}^n$ such that $K(x)$ is the scalar curvature of the new metric (up to a positive constant). When $K \equiv 1$ and $0$ is a nonremovable singularity, Caffarelli, Gidas and Spruck \cite{CGS} proved that any solution $u \in C^2 (B_1 \setminus \{ 0 \})$ of \eqref{scalar=01} is asymptotically radially symmetric near $0$ and moreover, there exists a global singular solution $u_0$ of
\begin{equation}\label{gloeq10}   
\begin{cases}
- \Delta u_0 = u_0^\frac{n + 2}{n - 2}, ~~~ u_0 > 0 ~~~ \textmd{in} ~ \R^n \setminus \{ 0 \}, \\
u_0 \in C^2 (\R^n \setminus \{ 0 \}) ~~~ \textmd{and} ~~~ \lim_{x \to 0} u_0 (x) = + \infty
\end{cases}
\end{equation}
such that
\begin{equation}\label{asy=glo}
u(x) = u_0 (x) (1 + o(1)) \quad  \textmd{as} ~  x \to 0. 
\end{equation} 
The global singular solutions of \eqref{gloeq10} are usually called the Fowler solutions or Delaunay-type solutions. Furthermore, Korevaar-Mazzeo-Pacard-Schoen \cite{KMPS} obtained the refined asymptotics in \eqref{asy=glo} and expanded singular solution $u$ to the first order, and Han-Li-Li \cite{HLL} established the expansions up to arbitrary orders. When $K \in C^1(B_1)$ is a positive nonconstant function with $K(0)=1$, Chen-Lin \cite{CL97} and Zhang \cite{Zhang02} showed, under certain flatness conditions of $K$ near $0$, that every singular solution $u$ of \eqref{scalar=01} satisfies
\begin{equation}\label{upper}  
u(x) \leq C |x|^{-\frac{n-2}{2}} \quad\quad  \textmd{for }  x  \in B_{1/2} \setminus \{ 0 \}. 
\end{equation} 
Based on this upper bound, Chen-Lin \cite{CL99a,L00} and Taliaferro-Zhang \cite{TZ06} further proved that $u$ is close to a Fowler solution $u_0$ of \eqref{gloeq10} as $x \to 0$. Du and Yang \cite{DY} have recently established a higher-order expansion for $u$ near the origin. Compared to these results in \cite{CL97,CL99a,L00,TZ06,Zhang02}, our Theorem \ref{m=1th} only assumes that $K \in L^\infty(B_{1})$ is nonnegative, and does not need to assume that $K$ is of class $C^1$ and satisfies any flatness conditions. More interestingly, estimate \eqref{upper} and asymptotics \eqref{asy=glo} are unstable under $C^1$ perturbations of $K$. In fact, Taliaferro \cite{Taliaferro2005} has showed that for $n \geq 6$ and for any $\varepsilon>0$, there exists a positive function $K \in C^1(B_1)$ satisfying $\|K - 1\|_{C^1(B_1)}  \leq \varepsilon$ such that \eqref{upper} does not hold, i.e., $u(x) \neq O\big( |x|^{- (n-2)/2 } \big)$ as $x \to 0$. Moreover, this singular solution can be constructed to be arbitrarily large near the origin. For all $n \geq 3$, \eqref{upper} is also unstable under $C^0$ perturbations of $K$ (see Taliaferro-Zhang \cite{TaliaferroZhang2003}). However, it follows from our Theorem \ref{m=1th} that the asymptotics \eqref{asy=01} is stable even under $C^0$ perturbations of $K$.

Our method also applies to isolated singularities of the following higher-order conformal $Q$-curvature equation 
\begin{equation}\label{meq}
(-\Delta)^{\frac{n}{2}} u= K(x) e^{u} \quad \textmd {in } B_{1} \setminus \{0\} \subset \mathbb{R}^{n}, 
\end{equation}
where $n \ge 3$ is an integer. Equation \eqref{meq} arises from the problem of finding a metric conformal to the flat metric of $\mathbb{R}^{n}$ such that $K(x)$ is the $Q$-curvature of the new metric (up to a positive constant ). For the study of prescribing $Q$-curvature problem, we may refer to \cite{CP,DM,HMM,JMMX,MS,WX1,WX2} and the references therein. Notice that \eqref{meq} is a nonlocal higher-order equation in odd dimensions. We first consider the even-dimensional case $n =2m$, and the second result of this paper is   
\begin{theorem}\label{mth} 
Suppose that $n=2m \geq 4$ is an even integer and $K \in L^\infty(B_{1})$ is nonnegative. Let $u \in C^{2m}(B_1 \setminus \{0\})$ be a solution of \eqref{meq}. If $u$ satisfies \eqref{la=abc} and 
\begin{equation}\label{lg}
    \int_{B_r \setminus \{0\} }|u(x)|dx=o\left( r^{2m-2} \right) \quad  \textmd{as } r \to 0, 
\end{equation}
then there exists a constant $\alpha > -2m$ and a H\"{o}lder continuous function $\varphi \in C^\gamma_{\textmd{loc}}(B_1)$ ($0< \gamma<1$) such that 
\begin{equation}
u(x) = \alpha \ln|x| + \varphi(x)  \quad \textmd{near the origin}.  
\end{equation}
\end{theorem}

When $m = 2$ and $K \equiv 1$, under the slightly stronger condition $|u(x)| = o(|x|^{-2})$, Guo-Liu \cite{GL} proved Theorem \ref{mth} and Guo-Liu-Wan \cite{GLW} further established the corresponding asymptotic expansion. Their proof is based on the spherical averaging technique and an ODE-type argument, which is entirely different from ours. In particular, our method applies to the general $Q$-curvature $K(x)$, as well as the higher-order case $m\ge 2$ and the nonlocal case $n\geq 3$. For the higher dimension $n > 2m$, the asymptotic behavior of conformal $Q$-curvature equation near isolated singularities has recently been studied in \cite{Ado,JX,JY,R}.

Next, we will consider the odd-dimensional case. Note that equation \eqref{meq} is nonlocal in this case. For the sake of convenience, we will understand the nonlocal operator $(-\Delta)^{\frac{n}{2}}$ in the sense of distributions. 
Denote 
$$
\mathcal{L}_s(\mathbb{R}^n):= \left\{ u \in L_{\textmd{loc}}^1(\mathbb{R}^n):  \int_{\mathbb{R}^n} \frac{|u(x)|}{1 + |x|^{n+2s}} dx < \infty  \right\}, \quad s>0. 
$$
We say that $u$ is a solution of \eqref{meq} if $u\in \mathcal{L}_{\frac{n}{2}}(\mathbb{R}^n)$, $K(x) e^u \in L_{\textmd{loc}}^1(B_1 \backslash \{0\})$, and $u$ satisfies \eqref{meq} in the following sense 
\begin{equation}\label{odd=02}
\int_{\mathbb{R}^n} u(x) (-\Delta)^{\frac{n}{2}} \varphi(x) dx = \int_{B_1} K(x) e^{u(x)} \varphi(x) dx ~~~~~ \textmd{for any} ~ \varphi \in C_c^\infty(B_1 \backslash \{0\}). 
\end{equation} 

\begin{theorem}\label{mth-odd} 
Suppose that $n\geq 3$ is odd and $K \in L^\infty(B_1)$ is nonnegative. Let $u \in \mathcal{L}_{\frac{n}{2}}(\mathbb{R}^n) \cap C^n(B_1 \setminus \{0\})$ be a solution of \eqref{meq}. If $u$ satisfies \eqref{la=abc} 
and 
\begin{equation}\label{lg=odd}
    \int_{B_r \setminus \{0\} }|u(x)|dx=o\left( r^{n-2} \right) \quad  \textmd{as } r \to 0, 
\end{equation}
then there exists a constant $\alpha > - n$ and a H\"{o}lder continuous function $\varphi \in C^\gamma_{\textmd{loc}}(B_1)$ ($0< \gamma<1$) such that 
\begin{equation}
u(x) = \alpha \ln|x| + \varphi(x)  \quad \textmd{near the origin}.  
\end{equation}
\end{theorem}

\begin{remark}\label{rem25=098} 
As pointed out in Remark \ref{rem25}, our Theorems \ref{mth} and \ref{mth-odd} can be applied to describe the asymptotic behavior for the conformal $Q$-curvature equation \eqref{meq} near infinity, as well as for unbounded function $K(x)$ with some growth restrictions at the origin. 
\end{remark}

Our method is unified for the conformal Gaussian curvature equation \eqref{m=1eq} and its higher-order analogue \eqref{meq} in any dimension $n \geq 3$. One of the main tools is the following characterization of isolated singularities for the Poisson equation. The case of nonnegative solutions has been studied in many literatures. For example, see Br\'ezis-Lions \cite{BL}, Ghergu-Moradifam-Taliaferro \cite{GMT}, Li \cite{Li96} and the references therein. To study the equations \eqref{m=1eq} and \eqref{meq}, we need to consider the solutions that may change sign.

\begin{theorem}\label{jdbs}  
Let $m$ and $n$ be positive integers with $n \geqslant 2m$. Suppose that $u \in C^{2m} (B_1 \setminus \{0\})$ is a solution of 
\begin{equation}\label{Poisson}
(-\Delta)^{m} u=f(x) \quad \textmd{in }  B_1 \setminus \{0\} \subset \mathbb{R}^{n}, 
\end{equation}
where $f$ satisfies
\begin{equation}\label{solutionl1}
\int_{B_1 \setminus \{0\}}|x|^{s}|f(x)|dx<+\infty 
\end{equation}
for some integer $0 \leq s \leq 2m-1$. If
\begin{equation}\label{solutionl123}
\int_{B_1 \setminus \{0\}} u^{+}(x)  dx < +\infty  
\end{equation}
with $u^{+}(x):= \max \left\{ u(x), 0 \right \} $, then there exist constants $a_{\beta} \in \mathbb{R}$ with $|\beta| \leq 2m-1$, such that 
\begin{equation}\label{udbseq}
u(x) = N(x) + h(x) + \sum_{|\beta| \leqslant 2m-1}a_{\beta} D^{\beta} \phi(x) \quad \textmd{for } x\in B_1 \setminus \{0\}, 
\end{equation}
where $N(x)$ is given in \eqref{Nx} and \eqref{Nx=763} below, $h \in C^{\infty}\left(B_{1} \right)$ is a solution of $(-\Delta)^{m} h=0$ in $B_{1} $, and $\phi$ is the fundamental solution of $(-\Delta)^{m}$ in $\mathbb{R}^{n}$. 
\end{theorem}

When $f \equiv 0$, Theorem \ref{jdbs} was proved by Futamura-Kishi-Mizuta \cite{FKM}. Here we will provide a different proof, the idea of which will also be applied to the nonlocal setting in Theorem \ref{mth-odd} (see Lemma \ref{jdbs=odd} in Section \ref{S03}). On the other hand, condition \eqref{solutionl1} allows $f$ not to be integrable in $B_1$, and thus Theorem \ref{jdbs} can be used to study \eqref{m=1eq} and \eqref{meq} when the finite volume assumption \eqref{la=abc} is not valid. We will discuss this problem in future work.

The paper is organized as follows. In Section \ref{S02}, we establish Theorem \ref{jdbs}. In Section \ref{S03}, we show Theorems \ref{m=1th}, \ref{mth} and \ref{mth-odd}.

\section{Representation formula}\label{S02}
Suppose that $m$ and $n$ are positive integers with $n \ge 2m$. A fundamental solution of $(-\Delta)^{m}$ in $\mathbb{R}^{n}$ is given by 
\begin{equation}\label{dyphi} 
\phi (x)=
c \begin{cases}
|x|^{2m-n}, ~~~ & n > 2m \geq 2,  \\ 
\ln  \frac{5}{|x|}, ~~~ & n=2m, 
\end{cases} 
\end{equation}
where $c=c(m,n)$ is a positive constant. Let $f$ satisfy \eqref{solutionl1} for some integer $0 \leq s \leq 2m-1$. 
Define $u^{+}(x) := \max \left\{u(x), 0\right \} $ and 
\begin{equation}\label{Nx}
N(x) : =\int_{\{|y| \leq 1\}}\psi(x, y)f(y) dy \quad \textmd{for } x \ne 0, 
\end{equation}
where 
\begin{equation}\label{Nx=763}
\psi(x, y) := \phi(x-y) - \sum_{|\beta| \leq s-1} \frac{(-y)^{\beta}}{\beta !} D^{\beta} \phi(x) \quad \textmd{for } x \ne 0, ~ y \neq x. 
\end{equation}
If $s=0$, then the summation in \eqref{Nx=763} does not appear.

\begin{proof}[Proof of Theorem \ref{jdbs}]  
Let $\psi(x, y)$ be defined as in \eqref{Nx=763}. First, we will prove that 
\begin{equation}\label{gjpsi}
\int_{B_1}|\psi(x,y)| dx \leqslant C|y|^{s},
\end{equation}
where $C>0$ is a constant.

For $|y| < \frac{|x|}{2}$ and $x \ne 0$, applying the Taylor formula we have 
\begin{equation}
    \psi(x,y)=\sum_{|\beta|=s} \frac{(-y)^{\beta}}{\beta !}  D^{\beta}\phi(x-\theta y),
\end{equation}
where $0 < \theta < 1$. A direct computation yields that 
\begin{equation}\label{psi01}
    |\psi(x,y)| \le C |y|^{s}|x|^{2m-n-s} \ln\frac{5}{|x|} \quad \textmd{for }  |y| < \frac{|x|}{2} \textmd{ and } x \ne 0,   
\end{equation}
where $C >0$ is a constant. Since $s \le 2m-1$, we have 
\[ 
\begin{aligned}
\int_{B_1}|\psi(x,y)|dx & \le C |y|^{s} \int_{\{ |y|<\frac{|x|}{2}, |x|<1\}}|x|^{2m-n-s} \ln\frac{5}{|x|}dx + \int_{\{|x| \le 2|y|, |x|<1\}}|\psi(x,y)|dx \\ 
& \le C \bigg[ |y|^{s} \int_{\{|x|<1\}}|x|^{2m-n-s} \ln\frac{5}{|x|}dx + \int_{\{ |x| \le 2|y| \}}|x-y|^{2m-n} \ln\frac{5}{|x-y|} dx \\
& ~~~~ + \sum_{|\beta| \le s-1}|y|^{|\beta|}\int_{\{ |x| \le 2|y| \} }|x|^{2m-n-|\beta|} \ln\frac{5}{|x|}dx \bigg] \\
&\leq C \bigg[ |y|^{s} + \int_{\{ |z| \le 3|y| \}}|z|^{2m-n} \ln\frac{5}{|z|} dz + |y|^{2m} \ln  \frac{5}{|y|} \bigg] \leq C |y|^{s}. 
\end{aligned}
\] 
Thus, \eqref{gjpsi} is proved.

Let $N(x)$ be defined as in \eqref{Nx}. Then, by the Fubini theorem, \eqref{solutionl1} and \eqref{gjpsi}, we get
\[ 
\begin{aligned}
\int_{B_1 } |N(x)|dx & \le  \int_{B_1} \bigg( \int_{\{|y| \leq 1\}} |\psi(x, y)| |f(y)| dy \bigg) dx  \\
& \le \int_{\{|y| \leq 1\}} \bigg( \int_{B_1}|\psi(x,y)|dx \bigg) |f(y)| dy \\
& \le C\int_{\{|y| \leq 1\}}|y|^{s}|f(y)|dy  < + \infty.  
\end{aligned}
\] 
Notice that 
\begin{equation}\label{laplace1}
(-\Delta)^m_x \psi(x,y) = 0 = (-\Delta)^m_y \psi(x,y) \quad \textmd{for } y \ne x, ~ x \ne 0. 
\end{equation} 
Let $r \in (0, 1)$ be arbitrarily fixed. For $r < |x| < 1$, we write $N(x)$ as
\[ 
N(x) = \int_{\{0< |y| < r/2 \}} \psi(x, y) f(y) dy + \int_{\{r/2  < |y| < 1 \}} \Big [\phi(x-y) - \sum_{|\beta| \leq s-1} \frac{(-y)^{\beta}}{\beta !} D^{\beta} \phi(x) \Big] f(y) dy. 
\]
By \eqref{psi01} and \eqref{solutionl1}, we can exchange the order of the first integral and the differentiation with respect to $x$. Thus, by \eqref{laplace1} we have 
\begin{equation}\label{NN0953}
(-\Delta)^m N (x) = f(x) \quad \textmd{for } r < |x| < 1. 
\end{equation} 
Since $r \in (0, 1)$ is arbitrary,  \eqref{NN0953} holds for all $x \in B_1 \setminus \{0\}$.

Now, we define 
\begin{equation}\label{dyv}
w(x):= u(x)-N(x) \quad \textmd{for } x \in B_1 \setminus \{0\}. 
\end{equation}
Then $w \in L_{\textmd{loc}}^{1}(B_1 \setminus \{0\}) $ and $w$ satisfies 
\begin{equation}
(-\Delta)^{m} w = 0 \quad \textmd{in  } B_1 \setminus \{0\} 
\end{equation}
in the sense of distributions. 
Moreover,  combining with \eqref{solutionl123}, we have  
\begin{equation}\label{w+09}
\int_{B_1 \setminus \{0\}} w(x)^{+} dx \le \int_{B_1 \setminus \{0\}} \big[ u(x)^{+} + |N(x)| \big] dx < +\infty. 
\end{equation}  
Let $\bar{w}(r)$ be the average of $w$ over the sphere $\partial B_r$, i.e., 
\[
\bar{w}(r) =  \frac{1}{s_n r^{n-1}} \int_{\partial B_r} w(x) ds, 
\]
where $s_n$ denotes the area of the unit sphere in $\mathbb{R}^n$. If $m=1$, then $\bar{w}$ satisfies 
\[
\frac{1}{r^{n-1}}( r^{n-1} \bar{w}_r )_r \equiv 0 \quad \textmd{in } (0, 1).
\]
This implies that 
\[
\bar{w}(r) = C_1 + C_2\begin{cases}
r^{2-n} \quad &\textmd{for } n\geq 3,\\
\ln \frac{1}{r} \quad &\textmd{for } n=2, 
\end{cases}
\]
and thus 
\begin{equation}\label{ww=001}
\int_0^1 \tau^{n-1} |\bar{w}(\tau)| d\tau < +\infty. 
\end{equation}    
For $m \geq 2$, one can also easily prove that \eqref{ww=001} is still true by mathematical induction. Hence, the limit 
\[
\lim_{r \to 0} \int_{B_1 \setminus B_r}  w(x)dx  = \lim_{r \to 0} s_n \int_{r}^1  \tau^{n-1} \bar{w}(\tau) d\tau ~ \textmd{exists}.  
\]
Combining with \eqref{w+09}, we obtain that 
\[
\int_{B_1 \setminus \{0\}}  w^{-}(x)dx = \lim_{r \to 0} \int_{B_1 \setminus B_r}  w^{+}(x)dx  - \lim_{r \to 0} \int_{B_1 \setminus B_r}  w(x)dx  ~ \textmd{exists},
\]
where $w^{-}(x):=\max\{ -w(x), 0\}$. Therefore, 
\begin{equation}\label{ww=00}
\int_{B_1 \setminus \{0\}} |w(x)| dx < +\infty.  
\end{equation}   
Thus $(-\Delta)^m w$ is a distribution in $\mathcal{D}^\prime(B_1)$ whose support is contained in $\{0\}$. By a classical result about distribution, there exists a nonnegative integer $K$ such that 
\[
(-\Delta)^m w =\sum_{|\beta| \leq K} a_\beta D^\beta \delta_0, 
\]
where $\delta_0$ is the Dirac measure at $0$. Let $h(x) := w(x) - \sum_{|\beta| \leqslant K} a_{\beta} D^{\beta} \phi(x) $. Then $h \in \mathcal{D}^\prime(B_1)$ satisfies \[
(-\Delta)^m h =  0 \quad \textmd{in } \mathcal{D}^\prime(B_1). 
\] 
From the classical regularity, we have that $h \in C^\infty(B_1)$ and $(-\Delta)^m h =  0$ in the  pointwise sense in $B_1$. Thus 
\[
\sum_{|\beta| \leqslant K} a_{\beta} D^{\beta} \phi(x) = w(x) -h(x) \in L_{\textmd{loc}}^1(B_1) 
\]
and so for any $k \geq 2m$,  we have $\sum_{|\beta| = k} a_{\beta} D^{\beta} \phi(x) \equiv 0$. Hence, we obtain that for some constants $a_{\beta} \in \mathbb{R}$ with $|\beta| \leq 2m-1$, 
\begin{equation}
w(x)=h(x) + \sum_{|\beta| \leqslant 2m-1}a_{\beta} D^{\beta} \phi(x)  \quad \textmd{for }  x \in B_1 \setminus \{0\}, 
\end{equation}
where $h \in C^{\infty}(B_{1})$ is the solution of $(-\Delta)^{m}h=0$ in $B_{1}$. Therefore, Theorem \ref{jdbs} follows from \eqref{dyv}.
\end{proof}  

\section{Isolated singularities}\label{S03} 

In this section, we prove Theorems \ref{m=1th}, \ref{mth} and \ref{mth-odd} based on Theorem \ref{jdbs}. We divide the section into three subsections to consider isolated singularities of Gaussian curvature equation, $Q$-curvature equation in even dimensions and $Q$-curvature equation in odd dimensions, respectively.

\subsection{Isolated singularities of Gaussian curvature equation}  

\begin{proof}[Proof of Theorem \ref{m=1th}]
In two dimensions, we recall that a fundamental solution of $-\Delta$ is given by $\phi(x) =\frac{1}{2\pi} \ln \frac{5}{|x|}$. Since $K\in L^{\infty} (B_1)$ and $u$ satisfies \eqref{la=abc}, using Theorem \ref{jdbs} with $s=0$ we obtain 
\begin{equation}\label{S03p=98}
u(x) = v(x) + h(x) + \sum_{|\beta| \leqslant 1} a_{\beta} D^{\beta} \phi(x) \quad \textmd{for }  x \in B_1 \setminus \{0\}, 
\end{equation}
where 
\begin{equation}\label{v90spo}
v(x) = \frac{1}{2\pi}\int_{\{ |y|<1\}}  \ln {\frac{5}{\left|x-y\right|}} K(y)e^{u(y)}dy,  
\end{equation}
and $h \in C^{\infty}\left(B_{1} \right)$ is a solution of $-\Delta h=0$ in $B_{1} $. Next, we will prove the following two conclusions.  
\begin{enumerate}[label = \rm(\roman*)]
\item $a_{\beta}= 0$ for $|\beta|=1$ and $a_0 < 4\pi$.  

\item $v \in C_{\textmd{loc}}^\gamma(B_1)$ for some $0< \gamma<1$. 
\end{enumerate}   

First, we verify the conclusion (i). By \eqref{la=abc}, we know 
\begin{equation}\label{keji}
\int_{B_{1/2} \setminus \{0\}} e^{v(x)} e^{h(x)} e^{\sum_{|\beta| \le 1}a_{\beta} D^{\beta} \phi (x)} dx < + \infty. 
\end{equation} 
The nonnegativity of $K$ implies $e^{v} \ge 1$ in $B_{1/2} \setminus \{0\}$, and $h \in C^{\infty}\left(B_{1}(0)\right)$ implies $M_1 < e^h < M_2$ in $B_{1/2}$ for two constants $M_1, M_2 >0$. Hence, \eqref{keji} yields 
\begin{equation}\label{keji2}
\int_{B_{1/2} \setminus \{0\}} e^{\sum_{|\beta| \le 1}a_{\beta} D^{\beta} \phi (x)}  dx < +\infty.    
\end{equation} 
Recall that $D^{(i,j)} \phi$ represents the $i$-th derivative of $\phi$ with respect to $x_1$ and the $j$-th derivative of $\phi$ with respect to $x_2$. Now we show that if $a_{(1,0)} \le 0$, then $a_{(0,1)}=0$. Consider the polar coordinates  
\[
\begin{cases} 
x_1 = r \cos \theta, \\ 
x_2=r \sin \theta. 
\end{cases} \quad 0\leq  r <1, ~ -\pi \leq \theta \leq \pi. 
\]
Set 
$$D_{1} = \left \{ (x_1, x_2) \in  B_{1/2} \setminus \{0\} \Big| 0 \le r  < \frac{1}{2}, ~ \frac{\pi}{4} \le \theta \le \frac{\pi}{2}\right \}.$$
By \eqref{keji2}, we obtain  
$$
\int_{D_1}\left( \frac{5}{\sqrt{x_1^2+x_2^2}} \right)^{\frac{1}{2 \pi} a_0} e^{\frac{1}{2 \pi}\left[a_{(1,0)}\left(-\frac{x_1}{x^2+x_2^2}\right)+a_{(0,1)}\left(-\frac{x_2}{x^2+x_2^2}\right)\right]} dx_1 dx_2< +\infty. 
$$
That is, 
$$
\int_{0}^{\frac{1}{2}}\int_{\frac{\pi}{4}}^{\frac{\pi}{2}}\left(\frac{5}{r}\right)^{\frac{1}{2 \pi} a_{0}} e^{\frac{1}{2 \pi}\left[a_{(1,0)}(-\frac{\cos\theta}{r})+a_{(0,1)}(-\frac{\sin\theta}{r})\right]} r d\theta dr<\infty.
$$
Using $a_{(1,0)} \le 0$ and $\frac{\pi}{4} \le \theta \le \frac{\pi}{2}$, we have 
$$
\int_{0}^{\frac{1}{2}}\int_{\frac{\pi}{4}}^{\frac{\pi}{2}}\left(\frac{5}{r}\right)^{\frac{1}{2 \pi} a_{0}} e^{\frac{1}{2 \pi}\left[a_{(0,1)}(-\frac{sin\theta}{r})\right]} r d\theta dr<\infty.
$$
Since $\sin\theta \ge \frac{\sqrt{2}}{2}$ in $D_{1}$, we must have $a_{(0,1)} \ge 0$ owing to 
$$
\int_{0}^{\frac{1}{2}} r^{a} e^{\frac{b}{r}} dr =+ \infty \quad  \forall~ a \in \mathbb{R}, ~ b>0. 
$$ 
Set 
$$D_{2} = \left \{ (x_1, x_2) \in B_{1/2} \setminus \{0\} \Big| 0 \le r < \frac{1}{2},  ~ -\frac{\pi}{2} \le \theta \le -\frac{\pi}{4}\right \}.$$
Then we can also get $a_{(0,1)} \le 0$. Hence, if $a_{(1,0)} \le 0$, then we have $a_{(0,1)} = 0$.

Similarly, if $a_{(1,0)} \ge 0$, then we can also show $a_{(0,1)} = 0$ by considering the integral in 
$$D_{3} = \left \{ (x_1, x_2) \in B_{1/2} \setminus \{0\} \Big| 0 \le r < \frac{1}{2},  ~ -\frac{3\pi}{4} \le \theta \le -\frac{\pi}{2}\right \}$$
and 
$$D_{4} = \left \{ (x_1, x_2) \in B_{1/2} \setminus \{0\} \Big| 0 \le r < \frac{1}{2},  ~ \frac{\pi}{2} \le \theta \le \frac{3\pi}{4}\right \},$$
respectively. Therefore, we always have $a_{(0,1)} = 0$. Thus, \eqref{keji2} implies  
$$
\int_{B_{1/2} \setminus \{0\} }\left( \frac{5}{\sqrt{x_1^2+x_2^2}} \right)^{\frac{1}{2 \pi} a_0} e^{\frac{1}{2 \pi}\left[a_{(1,0)}\left(-\frac{x_1}{x^2+x_2^2}\right)\right]} dx_1 dx_2< +\infty. 
$$
By direct computations, we obtain $a_{(1,0)} = 0$. Hence 
$$
\int_{B_{1/2} \setminus \{0\} }\left( \frac{5}{\sqrt{x_1^2+x_2^2}} \right)^{\frac{1}{2 \pi} a_0}  dx_1 dx_2 < +\infty, 
$$
which yields that $\frac{1}{2\pi}a_{0} < 2$, i.e., $a_0 < 4\pi$. This proves the first conclusion. 

Next, we prove the conclusion (ii). Let  
$$
\alpha := -\frac{1}{2\pi}a_0 > -2. 
$$ 
According to $K \in L^\infty(B_1)$ and \eqref{la=abc}, $v$ is a nonnegative solution of 
\begin{equation}\label{vdfc} 
-\Delta v = K(x)e^{u} = 5^{-\alpha} K(x)  e^{h(x)} |x|^{\alpha} e^{v} \quad \textmd{in } B_1 
\end{equation}
in the sense of distributions. Using $K \in L^\infty(B_1)$ and \eqref{la=abc} again, we have
\begin{equation}
    \int_{B_1 \setminus \{0\}} K(x)  e^{h(x)} |x|^{\alpha} e^{v(x)} dx<\infty.
\end{equation}
It follows from Br\'ezis-Merle \cite[Theorem 1, Corollary 1 and Remark 2]{BMF} that  $e^{v} \in L^{p}(B_{1/2})$ for each $p > 0$. Since $\alpha >-2$, choose some $\varepsilon>0$ such that $|x|^{\alpha} \in L^{1+\varepsilon}(B_{1/2})$. Hence, there exists some $p_0>1$ satisfying 
\[
K e^{h} |x|^{\alpha} e^{v}  \in L^{p_0} (B_{1/2}). 
\] 
By the classical $W^{2,p}$ estimates (see, e.g., Gilbarg-Trudinger \cite{GT}), we obtain $v \in W^{2,p_0}(B_{1/2})$. The Sobolev embedding theorem implies that $v \in C^\gamma(B_{1/2})$ for some $0< \gamma<1$. Hence, the conclusion (ii) holds. Combining with \eqref{S03p=98}, we now have
\[
u(x) = v(x) + h(x) + \alpha \ln |x| \quad \textmd{for }  x \in B_{1/2} \setminus \{0\}, 
\]
where $\alpha>-2$ and $\varphi := v+h \in C^\gamma(B_{1/2} )$ for some $0< \gamma<1$. The proof of Theorem \ref{m=1th} is completed.
\end{proof}

\subsection{Isolated singularities of $Q$-curvature equation: even dimensions}  

The method to prove Theorem \ref{m=1th} can also be applied to Theorem \ref{mth}. However, for $m\geq 2$, polyharmonic equations have higher-order singularities, and hence we need some new ideas to exclude these higher-order singularities.

\begin{proof}[Proof of Theorem \ref{mth}] 
In dimension $n=2m$, a fundamental solution of $(-\Delta)^{m}$ is given by $\phi(x) =c_m \ln \frac{5}{|x|}$, where $c_m>0$ is a constant. Suppose that $K\in L^{\infty} (B_1)$ and $u$ satisfies \eqref{la=abc}. By using Theorem \ref{jdbs} with $s=0$ we obtain  
\begin{equation}\label{S03pndju0}
u(x) = v(x) + h(x) + \sum_{|\beta| \leqslant 2m-1} a_{\beta} D^{\beta} \phi(x) \quad \textmd{for }  x \in B_1 \setminus \{0\}, 
\end{equation}
where 
\begin{equation}\label{v90sp9iuo0}
v(x) = c_m \int_{\{ |y|<1\}}  \ln {\frac{5}{\left|x-y\right|}} K(y)e^{u(y)}dy,  
\end{equation}
and $h \in C^{\infty}\left(B_{1} \right)$ is a solution of $(-\Delta)^m h=0$ in $B_{1} $. We will prove the following two conclusions. 
\begin{enumerate}[label = \rm(\roman*)]
\item $\sum_{1\leq |\beta| \leqslant 2m-1} a_{\beta} D^{\beta} \phi(x) \equiv 0$ and $a_0 < \frac{2m}{c_m}$.   

\item $v \in C_{\textmd{loc}}^\gamma(B_1)$ for some $0< \gamma<1$. 
\end{enumerate}    

First, we show the conclusion (i). By the Fubini theorem and \eqref{la=abc}, we obtain 
$$
\begin{aligned}
\int_{B_{r}} v(x)dx &\leq C \int_{B_{r}} \left( \int_{B_{1}} \ln \frac{5}{|x-y|}e^{u(y)}dy \right) dx\\
&\le C \int_{B_{1}} \left( \int_{B_{r}} \ln\frac{5}{|x-y|}dx \right) e^{u(y)} dy\\
&\le C  r^{2m} \ln\frac{1}{r}   \quad \textmd{for } 0< r < \frac{1}{2}. 
\end{aligned}
$$
This, together with \eqref{lg} and \eqref{S03pndju0}, implies  
\begin{equation}\label{near68hf}
\int_{B_{r}} \bigg| \sum_{|\beta|\le2m-1}  a_{\beta}D^{\beta}\phi(x) \bigg| dx=o(1)r^{2m-2} \quad \textmd{as } r \to 0.
\end{equation} 
Note that 
$$
\sum_{|\beta|=2m-1}a_{\beta}D^{\beta}\phi(x) = \frac{P_{2m-1}(x)}{|x|^{4m-2}}, 
$$
where $P_{2m-1}(x)$ is a homogeneous polynomial of order $2m-1$. If $P_{2m-1} \not \equiv 0$, then there exists $d_0 >0$ and a geodesic ball $U_0 \subset \mathbb{S}^{n-1}$ such that $|P_{2m-1}| \geq d_0>0$ in $U_0$. Set $x=r \theta$ with $\theta \in \mathbb{S}^{n-1}$, and $V_r=[0, r] \times U_0$. Then
$$
\int_{V_r} \bigg|  \sum_{|\beta|=2m-1}a_{\beta}D^{\beta}\phi(x) \bigg|  dx \geq d_1  r
$$
for some $d_1 >0$. On the other hand, a direct computation gives 
\[
\int_{B_{r}} \bigg| \sum_{|\beta|\le2m-2}  a_{\beta}D^{\beta}\phi(x) \bigg| dx \leq C r^2. 
\]
Thus, we have 
\[
\begin{aligned}
\int_{V_r} \bigg| \sum_{|\beta|\le2m-1}  a_{\beta}D^{\beta}\phi(x) \bigg| dx & \geq \int_{V_r} \bigg|  \sum_{|\beta|=2m-1}a_{\beta}D^{\beta}\phi(x) \bigg|  dx - \int_{V_r} \bigg| \sum_{|\beta|\le2m-2}  a_{\beta}D^{\beta}\phi(x) \bigg| dx \\
& \geq  \int_{V_r} \bigg|  \sum_{|\beta|=2m-1}a_{\beta}D^{\beta}\phi(x) \bigg|  dx - \int_{B_r} \bigg| \sum_{|\beta|\le2m-2}  a_{\beta}D^{\beta}\phi(x) \bigg| dx \\
& \geq d_1 r - C r^2\\
&\geq \frac{d_1}{2} r  \quad \textmd{for small } r > 0.  
\end{aligned}
\]
That contradicts \eqref{near68hf}. Hence, we have $P_{2m-1} \equiv 0$ and thus 
\begin{equation}\label{iki09}
\sum_{|\beta|=2m-1}a_{\beta}D^{\beta}\phi(x) \equiv 0.  
\end{equation} 
By repeating this process, we can obtain that 
\begin{equation}\label{i}
\sum_{|\beta|=j}a_{\beta}D^{\beta}\phi(x) \equiv 0 \quad \textmd{for each} ~ 2 \le j \le 2m-2. 
\end{equation}  
Now we show that $\sum_{|\beta|=1}a_{\beta}D^{\beta}\phi(x) \equiv  0$ as well. Using \eqref{la=abc}, \eqref{S03pndju0} and \eqref{i} we have 
\begin{equation}\label{kejimdy1}
\int_{B_{1/2} \setminus \{0\}}  e^{v(x)} e^{h(x)} e^{\sum_{|\beta| \le 1}a_{\beta} D^{\beta} \phi (x)} dx < +\infty.  
\end{equation}
Since $K(x)$ is non-negative and $h \in C^{\infty}\left(B_{1} \right)$, we have $e^{v} \ge 1$ and $e^h \geq M_1>0$ in $B_{1/2} \setminus \{0\}$. By \eqref{kejimdy1} we have 
\begin{equation}\label{keji2mdy1}
\int_{B_{1/2} \setminus \{0\}}  e^{\sum_{|\beta| \le 1}a_{\beta} D^{\beta} \phi (x)}  dx < +\infty.     
\end{equation}
Namely, 
\begin{equation}\label{keji2wmnj76}
\int_{B_{1/2} \setminus \{0\}}  \left( \frac{5}{|x|} \right)^{c_m a_0} e^{- c_m\sum_{i=1}^{2m} a_{i} \frac{x_i}{|x|^2}}  dx < +\infty,
\end{equation}
where $a_i = a_{(0, \cdots, 1, \cdots, 0)}$. Using a similar argument as in the proof of Theorem \ref{m=1th}, we obtain that $a_i=0$ for all $i=1, 2, \cdots, 2m$, and hence 
\begin{equation}\label{njd63}
\sum_{|\beta|= 1} a_{\beta} D^{\beta} \phi(x) \equiv 0.
\end{equation}
From \eqref{keji2wmnj76} we also have 
\[
\int_{B_{1/2} \setminus \{0\}}  \left( \frac{5}{|x|} \right)^{c_m a_0}  dx < +\infty, 
\]
which implies $c_m a_0 < 2m$. This, together with \eqref{i} and \eqref{njd63}, gives the proof of conclusion (i).

Next, we prove the conclusion (ii). Let  
$$
\alpha := -c_m a_0 > -2m.  
$$ 
According to $K \in L^\infty(B_1)$ and \eqref{la=abc}, $v$ is a nonnegative solution of 
\begin{equation}\label{vdfcm2} 
(-\Delta)^m v = K(x)e^{u} = 5^{-\alpha} K(x)  e^{h(x)} |x|^{\alpha} e^{v} \quad \textmd{in } B_1 
\end{equation}
in the sense of distributions. Using $K \in L^\infty(B_1)$ and \eqref{la=abc} again, we have
\begin{equation}
    \int_{B_1 \setminus \{0\}} K(x)  e^{h(x)} |x|^{\alpha} e^{v(x)} dx<\infty.
\end{equation}

To estimate $v$, we need the following result. Its proof is inspired by Br\'ezis-Merle \cite{BMF}. 
\begin{lemma}\label{lemma3-1}
Suppose that $m$ is a positive integer. Let $v$ be a solution of   
\begin{equation}\label{gjbsfc}
(-\Delta)^{m} v = f(x) \quad \textmd{in }  B_1 \subset \mathbb{R}^{2m}, 
\end{equation} 
with $f \in L^1_{\textmd{loc}}(B_1)$. Then for every $ p > 1$,
\begin{equation}\label{eulpkj}
e^{p|v|} \in L^1_{\textmd{loc}}(B_1). 
\end{equation}    
\end{lemma} 

\begin{proof} Let $r \in (0, 1)$ be arbitrarily fixed. Consider $v = w+h$, where  
\begin{equation}\label{w}
\begin{cases}
(-\Delta)^{m}w = f(x) \quad & \textmd{in }  B_r, \\
(-\Delta)^{i}w = 0 \quad & \textmd{on }  \partial B_r, ~ i=0, 1, \cdots, m-1 
\end{cases}
\end{equation}
and
\begin{equation}\label{h}
\begin{cases}
(-\Delta)^{m} h = 0 \quad & \textmd{in } B_r, \\
(-\Delta)^{i} h = (-\Delta)^{i}v  \quad &\textmd{on }\partial B_r, ~ i=0, 1, \cdots, m-1.  
\end{cases}
\end{equation} 
Let 
\[
N(x)= c_m \int_{B_r}  \ln\frac{5}{|x-y|} |f(y)|dy, \quad x\in \mathbb{R}^{2m}. 
\]
Then
\[
(-\Delta)^m N = |\bar{f}(x)| \quad \textmd{in }  \mathbb{R}^{2m}, 
\]
where $\overline{f}$ is the extension of $f$ to zero outside $B_r$. Note that $N(x) \geq 0$ for $x \in B_r$ since $\frac{5}{|x -y|} \geq 1$, $\forall x, y \in B_r$. Moreover, for $i=1,...,m-1$, 
\[
(-\Delta)^{i} N(x)= c_m^{\prime} \int_{B_r} |x-y|^{-2i} |f(y)|dy \geq  0  \quad \textmd{for }  x \in B_r. 
\] 
By the maximum principle, we obtain 
\begin{equation}\label{c2dh6j}
|w| \leq N \quad \textmd{in } B_r.  
\end{equation} 
Thus, for any $k>0$ satisfying $k < k_0:=\frac{2m}{c_m}$, we have
\begin{equation}\label{c9j26j}
\int_{B_r} e^{\frac{k |w(x)|}{\|f\|_{1}}} dx \leq \int_{B_r} e^{\frac{k N(x)}{\|f\|_{1}}} dx, 
\end{equation} 
where $\|f\|_{1}=\int_{B_r} |f(x)| dx$. Using the Jensen inequality, we get 
\[
\begin{aligned}
\int_{B_r} e^{\frac{k N(x)}{\|f\|_{1}}} dx & =  \int_{B_r} e^{\int_{B_r}  k c_m \ln\frac{5}{|x-y|} \frac{|f(y)|}{\|f\|_{1}}dy} dx \\
& \le \int_{B_r} \left( {\int_{B_r}  \Big(\frac{5}{|x-y|} \Big)^{k c_m} \frac{|f(y)|}{\|f\|_{1}}dy} \right)dx \\
& = \int_{B_r} \left( {\int_{B_r}  \Big(\frac{5}{|x-y|} \Big)^{k c_m} } dx\right) \frac{|f(y)|}{\|f\|_{1}}dy \\
& \le {\int_{B_r}  \Big(\frac{5}{|x|} \Big)^{k c_m} } dx \\
&=\frac{5^{k c_m}}{2m - kc_m} r^{2m - k c_m} < +\infty. \\
\end{aligned}    
\]
Hence, we obtain that for any $k \in (0, k_0) $ with $k_0=\frac{2m}{c_m} >0$, 
\begin{equation}\label{cblpgj}
\int_{B_r} e^{\frac{k |w(x)|}{\|f\|_{1}}} dx \leq \frac{5^{k c_m}}{2m - kc_m} r^{2m - k c_m} < +\infty. 
\end{equation}
In particular, 
\begin{equation}\label{l48gj}
\int_{B_r} e^{\frac{m |w(x)|}{c_m \|f\|_{1}}} dx < +\infty.  
\end{equation}

For  every $p>1$, let $0 < \varepsilon < \frac{m}{c_m p}$. We may split $f$ as $f=f_1 + f_2$ with $\|f_1\|_1 < \varepsilon$ and $f_2 \in L^\infty(B_r)$. Write $w=w_1 +w_2$, where $w_1$ and $w_2$ are solutions of  
\begin{equation}\label{wi=01}
\begin{cases}
(-\Delta)^{m}w_{1} = f_{1}(x) \quad &\textmd{in } B_r, \\
(-\Delta)^{i}w_{1} = 0 \quad &\textmd{on } \partial B_r,  ~i=0, 1, \cdots, m-1 
\end{cases}
\end{equation}
and
\begin{equation}\label{wi=02}
\begin{cases}
(-\Delta)^{m}w_{2} = f_{2}(x) \quad &\textmd{in } B_r, \\
(-\Delta)^{i}w_{2} = 0 \quad &\textmd{on } \partial B_r,  ~i=0, 1, \cdots, m-1, 
\end{cases}
\end{equation}
respectively. Applying estimate \eqref{l48gj} to $w_1$, we find $\int_{B_r} e^{\frac{m |w_1(x)|}{c_m \|f_1\|_{1}}} dx < +\infty$ and thus 
\[
\int_{B_r} e^{p|w_1(x)|} dx \leq \int_{B_r} e^{\frac{m |w_1(x)|}{c_m \varepsilon}} dx \leq \int_{B_r} e^{\frac{m |w_1(x)|}{c_m \|f_1\|_{1}}} dx < +\infty. 
\]
Since $h, w_2 \in L^\infty(B_r)$ and $|v| \leq |w_1| + |w_2| + |h|$, we have for any $p>1$ that 
\[
\int_{B_r} e^{p|v(x)|} dx < +\infty. 
\]
This establishes Lemma \ref{lemma3-1}.  
\end{proof}

Applying Lemma \ref{lemma3-1} to \eqref{vdfcm2}, we obtain $e^{v} \in L^{p}(B_{1/2})$ for every $p > 1$. Since $\alpha >-2m$, one can choose some $\varepsilon >0$ such that $|x|^{\alpha} \in L^{1+\varepsilon}(B_{1/2})$. Hence, there exists some $p_0>1$ such that 
\[
K e^{h} |x|^{\alpha} e^{v}  \in L^{p_0} (B_{1/2}). 
\] 
By the classical $W^{2m,p}$ estimates, we obtain $v \in W^{2m, p_0}(B_{1/2})$. The Sobolev embedding theorem implies that $v \in C^\gamma(B_{1/2})$ for some $0< \gamma<1$. Hence, the conclusion (ii) holds. Together with \eqref{S03pndju0}, we get 
\[
u(x) = v(x) + h(x) + \alpha \ln |x| \quad \textmd{for }  x \in B_{1/2} \setminus \{0\}, 
\]
where $\alpha>-2m$ and $\varphi := v+h \in C^\gamma(B_{1/2} )$ for some $0< \gamma<1$. The proof of Theorem \ref{mth} is completed. 
\end{proof}

\subsection{Isolated singularities of $Q$-curvature equation: odd dimensions} 

The method above can also be applied to the odd-dimensional case in Theorem \ref{mth-odd}. Since equation \eqref{meq} is nonlocal, we establish the following nonlocal version of Theorem \ref{jdbs}.    

\begin{lemma}\label{jdbs=odd}  
Suppose that $n\geq 3$ is an odd integer. Let $u \in \mathcal{L}_{\frac{n}{2}}(\mathbb{R}^n) \cap C^{n} (B_1 \setminus \{0\})$ be a solution of 
\begin{equation}\label{Poisson-odd}
(-\Delta)^{\frac{n}{2}} u=f(x) \quad \textmd{in }  B_1 \setminus \{0\} \subset \mathbb{R}^{n}, 
\end{equation}
where $f \in L^1(B_1)$. Then there exist constants $a_{\beta} \in \mathbb{R}$ with $|\beta| \leq n-1$, such that 
\begin{equation}\label{udbseqodd}
u(x) = v(x) + h(x) + \sum_{|\beta| \leqslant n-1}a_{\beta} D^{\beta} \phi(x) \quad \textmd{for }  B_1 \setminus \{0\}, 
\end{equation}
where $h \in \mathcal{L}_{\frac{n}{2}}(\mathbb{R}^n) \cap C^{\infty} \left(B_{1} \right)$ satisfies $(-\Delta)^{\frac{n}{2}} h=0$ in $B_{1} $, $\phi(x)= c_n  \ln {\frac{5}{\left|x\right|}}$ is a fundamental solution of $(-\Delta)^{\frac{n}{2}}$ in $\mathbb{R}^{n}$, and $v$ is defined by 
\begin{equation}\label{vmj=odd}
v(x) = c_n \int_{\{ |y|<1\}}  \ln {\frac{5}{\left|x-y\right|}} f(y) dy,  \quad x \in \mathbb{R}^n. 
\end{equation}
\end{lemma}     

\begin{proof}
Let $v$ be defined as in \eqref{vmj=odd}. Then $v \in L_{\textmd{loc}}^1(\mathbb{R}^n)$ and 
\[
|v(x)| \leq C \ln |x| \quad \textmd{for large } |x|. 
\]
This implies that $v \in \mathcal{L}_{\frac{n}{2}}(\mathbb{R}^n) $ and 
\[
\begin{aligned}
\int_{\mathbb{R}^n} v(x) (-\Delta)^{\frac{n}{2}} \varphi(x) dx &=   \int_{\{ |y|<1\}}  \left( \int_{\mathbb{R}^n} c_n \ln {\frac{5}{\left|x-y\right|}} (-\Delta)^{\frac{n}{2}} \varphi(x) dx \right) f(y) dy \\
& = \int_{\{ |y|<1\}}  \varphi(y)   f(y) dy \quad \textmd{for any } \varphi \in C_c^\infty(B_1). 
\end{aligned}
\]
That is, $v$ is a solution of 
\[
(-\Delta)^{\frac{n}{2}} v = f \quad \textmd{in }  B_1. 
\]
Let $w = u -v $. Then $w \in \mathcal{L}_{\frac{n}{2}}(\mathbb{R}^n) $ satisfies 
\[
(-\Delta)^{\frac{n}{2}} w = 0 \quad \textmd{in }  B_1 \setminus \{0\}. 
\]
Notice that $(-\Delta)^{\frac{n}{2}} w$ is a tempered distribution in $\mathcal{S}^\prime(\mathbb{R}^n)$. We restrict $(-\Delta)^{\frac{n}{2}} w$ to $B_1$, and the resulting distribution is denoted as $T$ whose support is contained in $\{0\}$. By a classical result about distribution, there exists a nonnegative integer $K$ such that 
\[
T =\sum_{|\beta| \leq K} a_\beta D^\beta \delta_0, 
\]
where $\delta_0$ is the Dirac measure at $0$. Let $h(x) := w(x) - \sum_{|\beta| \leqslant K} a_{\beta} D^{\beta} \phi(x) $. Then $h \in \mathcal{S}^\prime(\mathbb{R}^n)$ satisfies \[
(-\Delta)^{\frac{n}{2}} h =  0 \quad \textmd{in } B_1. 
\] 
By the regularity result (see, e.g., \cite{Garofalo,Sil}), we know that $h \in C^\infty(B_1)$. Thus 
\[
\sum_{|\beta| \leqslant K} a_{\beta} D^{\beta} \phi(x) = w(x) -h(x) \in L_{\textmd{loc}}^1(B_1) 
\]
and so for any $k \geq n$,  we have $\sum_{|\beta| = k} a_{\beta} D^{\beta} \phi(x) \equiv 0$. Hence, we obtain that for some constants $a_{\beta} \in \mathbb{R}$ with $|\beta| \leq n-1$, 
\begin{equation}
w(x)=h(x) + \sum_{|\beta| \leqslant n-1}a_{\beta} D^{\beta} \phi(x)  \quad \textmd{for }  x \in B_1 \setminus \{0\}, 
\end{equation}
where $h \in \in \mathcal{L}_{\frac{n}{2}}(\mathbb{R}^n) \cap C^{\infty}(B_{1})$ is the solution of $(-\Delta)^{\frac{n}{2}}h=0$ in $B_{1}$. This completes the proof of Lemma \ref{jdbs=odd}. 
\end{proof}

\begin{proof}[Proof of Theorem \ref{mth-odd}] 
In $\mathbb{R}^n$, a fundamental solution of $(-\Delta)^{\frac{n}{2}}$ is given by $\phi(x) =c_n \ln \frac{5}{|x|}$, where $c_n>0$ is a constant. Suppose that $K\in L^{\infty} (B_1)$ is nonnegative and\ $u \in \mathcal{L}_{\frac{n}{2}}(\mathbb{R}^n) \cap C^n(B_1 \setminus \{0\})$ satisfies \eqref{la=abc}. 
By Lemma \ref{jdbs=odd}, we have 
\begin{equation}\label{jkdvi7}
u(x) = v(x) + h(x) + \sum_{|\beta| \leqslant n-1} a_{\beta} D^{\beta} \phi(x) \quad \textmd{for }  x \in B_1 \setminus \{0\}, 
\end{equation}
where $h \in C^{\infty}\left(B_{1} \right)$ and
\begin{equation}\label{3jh7=odd}
v(x) = c_n \int_{\{ |y|<1\}}  \ln {\frac{5}{\left|x-y\right|}} K(y) e^{u(y)}dy.  
\end{equation}
We will prove the following two conclusions. 
\begin{enumerate}[label = \rm(\roman*)]
\item $\sum_{1\leq |\beta| \leqslant n -1} a_{\beta} D^{\beta} \phi(x) \equiv 0$ and $a_0 < \frac{n}{c_n}$.   

\item $v \in C_{\textmd{loc}}^\gamma(B_1)$ for some $0< \gamma<1$. 
\end{enumerate}

The proof of the conclusion (i) is similar to that of Theorem \ref{mth}, so we omit the details. Next, we prove the conclusion (ii).

We first show that $e^{p v} \in L^1(B_1)$ for any $ p > 1$. For any $p>1$, take $0 < \varepsilon < \frac{n}{2c_n p}$. We may split $K e^u$ as $K e^u=f_1 + f_2$ with $\|f_1\|_1 < \varepsilon$ and $f_2 \in L^\infty(B_1)$. Then $v$ can be written as 
$v=v_1 + v_2$, where
\[
v_i(x) := c_n \int_{\{ |y|<1\}}  \ln {\frac{5}{\left|x-y\right|}} f_i(y) dy, \quad i=1, 2.
\]
By the Jensen inequality, we obtain 
\[
\begin{aligned}
\int_{B_1} e^{p |v_1(x)| } dx & \leq \int_{B_1} e^{\int_{B_1}  \frac{n}{2} \ln\frac{5}{|x-y|} \frac{|f_1(y)|}{\|f_1\|_1} dy} dx \\
& \le \int_{B_1} \left( {\int_{B_1}  \Big(\frac{5}{|x-y|} \Big)^{\frac{n}{2}} \frac{|f_1(y)|}{\|f_1\|_1} dy} \right)dx \\
& = \int_{B_1} \left( {\int_{B_1}  \Big(\frac{5}{|x-y|} \Big)^{\frac{n}{2}} } dx\right) \frac{|f_1(y)|}{\|f_1\|_1} dy \\
& \le {\int_{B_1}  \Big(\frac{5}{|x|} \Big)^{\frac{n}{2}} } dx < +\infty. 
\end{aligned}    
\]
On the other hand, we also have  
\[
\|v_2\|_{L^\infty(B_1)} \leq C \|\ln x\|_{L^1(B_1)} \|f_2\|_{L^\infty(B_1) } < +\infty.  
\]
Thus, for any $ p > 1$, we have $e^{p v} \in L^1 (B_1)$. Let  
$$
\alpha := -c_n a_0 > -n.  
$$ 
Then one can choose some $\varepsilon_0 >0$ such that $|x|^{\alpha} \in L^{1+\varepsilon_0}(B_1)$. Hence, there exists some $p_0>1$ such that 
\[
Ke^u = K e^{h} |x|^{\alpha} e^{v}  \in L^{p_0} (B_{1/2}). 
\] 
By the classical potential estimate (see, e.g., Gilbarg-Trudinger \cite{GT}), we obtain that $v \in C^\gamma(B_{1/2})$ for some $0< \gamma<1$. The conclusion (ii) is proved. Hence, we get 
\[
u(x) = v(x) + h(x) + \alpha \ln |x| \quad \textmd{for }  x \in B_{1/2} \setminus \{0\}, 
\]
where $\alpha>-n$ and $\varphi := v+h \in C^\gamma(B_{1/2} )$ for some $0< \gamma<1$. The proof of Theorem \ref{mth-odd} is completed. 
\end{proof}

\bigskip

\noindent  Hui Yang 

\noindent School of Mathematical Sciences, CMA-Shanghai \\
Shanghai Jiao Tong University, Shanghai 200240, China \\[1mm] 
Email: \textsf{hui-yang@sjtu.edu.cn}

\bigskip

\noindent  Ronghao Yang

\noindent School of Mathematical Sciences \\
Shanghai Jiao Tong University, Shanghai 200240, China \\[1mm] 
Email: \textsf{ylqfyrh5218@sjtu.edu.cn}

\end{document}